\documentclass[12pt]{amsart}

\usepackage[english]{babel}
\usepackage[utf8x]{inputenc}
\usepackage[T1]{fontenc}

\usepackage[margin=1.5in]{geometry}

\usepackage{amsmath}
\usepackage{amstext}
\usepackage{amsfonts}
\usepackage{amssymb}
\usepackage{amsthm}
\usepackage{amsrefs}
\usepackage{mathtools}
\usepackage{dsfont}
\usepackage{color}
\usepackage{graphicx}
\usepackage[colorinlistoftodos]{todonotes}
\usepackage[colorlinks=true, allcolors=blue]{hyperref}
\usepackage{float}
\usepackage{mathrsfs}
\allowdisplaybreaks

\newtheorem{theorem}{Theorem}

\newtheorem{prop}[theorem]{Proposition}

\newtheorem{cor}[theorem]{Corollary}

\newtheorem{thm}[theorem]{Theorem}
\newtheorem{lem}[theorem]{Lemma}

\newtheorem{theo}{Theorem}

\newtheorem{coro}{Corollary}

\newtheorem*{cor*}{Corollary}
\newtheorem*{thm*}{Theorem}
\newtheorem*{lem*}{Lemma}

\newtheorem*{prop*}{Proposition}

\theoremstyle{definition}

\newtheorem{defn}[theorem]{Definition}

\newtheorem{ex}[theorem]{Example}
\newtheorem{question}{Question}

\newtheorem*{defn*}{Definition}

\theoremstyle{remark}

\newcommand{\pr}{\operatorname{Prob}}

\newcommand{\act}{\curvearrowright}

\newcommand{\cX}{\mathcal{X}}

\newcommand{\bN}{{\mathbb{N}}}
\newcommand{\bZ}{{\mathbb{Z}}}

\newcommand{\bR}{{\mathbb{R}}}

\newcommand{\g}{\gamma}

\newcommand{\om}{\omega}

\newcommand{\id}{\operatorname{id}}

\newcommand{\Sub}{\operatorname{Sub}}

\newcommand{\C}{\mathbb C}
\newcommand{\G}{\Gamma}

\newcommand{\ep}{\varepsilon}

\newcommand{\one}{\mathds{1}}

\usepackage[nodisplayskipstretch]{setspace}

\allowdisplaybreaks

\title[]{Growth conditions for topological freeness}

\author[N. Alam]{Nazmul Alam}
\address{Nazmul Alam, University of Houston\\ USA}
\email{malam24@cougarnet.uh.edu}

\author[J. Gondek]{Joseph Gondek}
\address{Joseph Gondek, University of Oxford\\ UK}
\email{joseph.gondek@stx.ox.ac.uk}

\author[M. Kalantar]{Mehrdad Kalantar}
\address{Mehrdad Kalantar\\ University of Oxford\\ UK}
\email{mehrdad.kalantar@maths.ox.ac.uk}

\author[R. Pham]{Randy Pham}
\address{Randy Pham, University of Oxford\\ UK}
\email{randy.pham@trinity.ox.ac.uk}

\date{}

\begin{document}

\begin{abstract}
Given a finitely generated group $\G$, a non-trivial element $g\in \G$, and a minimal action $\G\act \mathcal{X}$ on a compact space $\mathcal{X}$, with amenable neighborhood stabilizers, we prove sufficient conditions in terms of various growth/decay functions for topological freeness of the action of $g$ on $\mathcal{X}$. We apply our results to the case of the Furstenberg boundary action of $\G$, to conclude \textup{C}$^*$-simplicity under (sub-)rapid decay conditions. In this context, we also give a description of the support of stationary states on the reduced $\textup{C}^*$-algebra of $\G$.
\end{abstract}

\maketitle

\section{Introduction and Main results} 

\let\thefootnote\relax\footnotetext{MK was supported by the NSF grant DMS-2155162, and The Simons Foundation grant SFI-MPS-TSM-00014307.
For the purpose of open access, the authors have applied a CC-BY license to any author accepted manuscript arising from this submission.}

The dynamics of boundary actions of discrete groups encode important information about the various structural properties of $\textup{C}^*$-algebras they generate. Admitting free boundary actions implies simplicity of the reduced $\textup{C}^*$-algebra \cite{KK14}; if the free boundary action admits unique stationary measures, one has further rigidity results for measurable actions of $\G$ \cite{hart-kalantar-1}; and, if $\G$ admits free extreme boundary actions, then its reduced $\textup{C}^*$-algebra has strict comparison \cite{Oz25}. 

We recall from \cite{Furstenberg73} that a compact $\G$-space $\cX$ is called a $\G$-boundary (in the sense of Furstenberg) if for every probability measure $\nu\in \pr(\cX)$ and every $x\in \mathcal{X}$, there is a net $(g_i)$ of elements $g_i\in \G$ such that $g_i\nu\to \delta_x$ in the weak* topology in $\pr(\cX)$.
Every group $\G$ has a unique (up to $\G$-equivariant homeomorphisms) universal $\G$-boundary $\partial_F\G$, in the sense that every $\G$-boundary $\cX$ is a continuous $\G$-equivariant image of $\partial_F\G$ (\cite{Furstenberg73}*{Proposition 4.6}). This universal boundary $\partial_F\G$ is called the Furstenberg boundary of $\G$. 

There are several (indirect) sufficient conditions to ensure freeness of the action $\G\act\partial_F\G$ for a given group $\G$ (e.g. lack of amenable normalish subgroups \cite{BKKO2017} or non-trivial amenable URS \cite{Ken15}, existence of continuous actions with non-amenable rigid stabilizers \cite{LeB-MBon18}, ...).
In this paper, we address this problem rather locally. We give sufficient conditions for the freeness of the action of a single given element $g\in \G$ on $\partial_F\G$; that is, to determine when the fixed-point set ${\rm Fix}(g)=\{x\in \partial_F\G : gx=x\}$ is empty.

In fact, we prove our result in a more general setting, for minimal actions of $\G$ on compact spaces $\mathcal{X}$ such that neighborhood stabilizer subgroups $\G_x^o:=\{g\in \G : gy=y ~ \text{for all $y$ in an open neighborhood of $x$}\}$, are amenable for all $x\in \cX$.

Our conditions are based on comparison of various growth/decay rates.

\begin{defn}\label{SRD}
Let $\G$ be a finitely generated group, with a fixed symmetric generating set. 
For each $n\in\bN$, define
\begin{equation*}
\rho(n) = \max\left\{\frac{\|\lambda(f)\|}{\|f\|_2}: 0\neq f\in \C\Gamma,\, {\rm supp}(f)\subseteq B_n(\G)\right\} .
\end{equation*}
Given a subset $A\subseteq\G$, we denote its growth function by
\begin{equation*}\alpha_{\scriptscriptstyle A}(n) = |A\cap B_n(\G)| ,
\end{equation*}
and given a subgroup $H\le\G$, we denote its \emph{co-growth} function by
\begin{equation*}
\beta_{\scriptscriptstyle H}(n) = |B_n(G/H)| ,\end{equation*}
where $B_n(G/H)$ is the ball of radius $n$ around the trivial coset in the Schreier graph $G/H$.
\end{defn}

The group $\G$ is said to have the \emph{rapid decay property} (RD) if $\rho(n)$ grows polynomially \cite{joli}. 
We say $\G$ has \emph{sub-rapid decay} (abbreviated SRD) if $\rho(n)$ grows subexponentially.  
Any group with intermediate growth has SRD but not RD, and the direct product of an SRD group and an RD group is SRD \cite{KS2025}*{Lemma 1}. (See also \cite{SriGreg-SRD} for more details on this property and more examples.)

Our first main result is a sufficient condition for topological freeness of the action of an element $g\in \G$ in terms of the above rates. Given functions $\phi, \psi: \bN\to\bR$, we write $\phi=o(\psi)$ if $\lim_{n\to\infty}\phi(n)/\psi(n)= 0$ and we write $\phi=\underline{o}(\psi)$ if $\liminf_{n\to\infty}\phi(n)/\psi(n)= 0$.

\begin{theo}\label{thm:freeness}
Let $\G$ be a finitely generated group, let $g\in\Gamma$ be a non-central element, and denote by $C=C_\G(g)$ the centralizer of $g$ in $\G$. If 
\begin{equation}\label{eqn}
\lim_n\rho(3n)^2\,\alpha_{\scriptscriptstyle C\,\cap\, \g C\g^{-1}}(2n) /\alpha_{\scriptscriptstyle C}(n) = 0 \tag{$\star$}
\end{equation}
for all $\g\notin C$,
then $g$ acts topologically freely on every compact minimal $\Gamma$-space $\mathcal{X}$ with amenable neighborhood stabilizers.
\end{theo}

Theorem \ref{thm:freeness} has an interpretation in the language of Glasner and Weiss' \textit{uniformly recurrent subgroups} (URS's), introduced in \cite{gw_urs}.

\begin{coro}\label{urs}
Let $\Gamma$ be a finitely generated group, and let $g\in\Gamma$ satisfy the hypotheses of Theorem \ref{thm:freeness}. If $\mathcal{H}\subset \Sub(\Gamma)$ is an amenable URS, then $g\not\in H$ for all $H\in\mathcal{H}$.
\end{coro}

Applying Theorem \ref{thm:freeness} to groups with (sub-)rapid decay, we get

\begin{coro}\label{cor:srd}
Let $\G$ be a finitely generated non-amenable group with sub-rapid decay. Let $g\in\G$ be a non-central element, and denote $C=C_\G(g)$. 
If $C\,\cap\, \g C\g^{-1}$ grows subexponentially for all $\g\notin C$,
then $g$ acts topologically freely on any compact minimal $\G$-space $\mathcal{X}$ with amenable neighborhood stabilizers.
\end{coro}

\medskip

A general sufficient condition for topological freeness of the action of $g\in\G$ on any compact minimal $\Gamma$-space $\mathcal{X}$ with amenable neighborhood stabilizers is that for some $\mu\in \pr(\G)$, and every $\mu$-stationary state $\tau$ on $\textup{C}^*_r(\G)$, $\tau(\lambda_g)=0$ (see Lemma~\ref{lem:station-topfree} below). The notion of stationary state was studied in \cite{hart-kalantar-1}, where a
characterization of $\textup{C}^*$-simplicity in terms of unique stationarity was given: namely, it was proven that a group $\G$ is $\textup{C}^*$-simple if and only if for some probability measure $\mu\in {\rm Prob}(\G)$, the canonical trace is the unique $\mu$-stationary state on $\textup{C}^*_r(\G)$. In contrast to the case of traces, there is no general (purely group theoretic) description of stationary states in terms of their support.

Our next result gives such a description in terms of the random walks on Schreier graphs with respect to centralizers of elements of $\Gamma$. This is natural, since $\mu$-stationary measures and states are essentially objects associated to the random walk defined by $\mu$ \cite{Furstenberg73}.

\begin{theo}\label{thm:Gong}
Let $\G$ be a finitely generated group, let $\mu\in{\rm Prob}(\G)$ be a symmetric measure supported on the ball of radius $N$. Let $g\in\Gamma$, and denote by $\kappa_n$ the probability of the return to the origin after $n$ steps in the $\mu$-random walk on the Schreier graph $\G/C_\G(g)$.
If 
\[
\liminf_n \rho\big(2nN+|g|\big)\,\kappa_{2n}^{1/4} = 0 ,
\]
then every $\mu$-stationary state on $\textup{C}^*_r(\G)$ vanishes at $g$.
\end{theo}

Again, we record the conclusion in the case of groups with the (sub-)rapid decay property.

\begin{coro}\label{cor:stn}
Let $\Gamma$ be a group with \textup{RD} \textup{(}resp. \textup{SRD}\textup{)}, let $\mu\in{\rm Prob}(\G)$ be finitely supported and symmetric, and fix $g\in\Gamma$. If the probability of return to the origin in the $\mu$-random walk on the Schreier graph $\G/C_\G(g)$ decays superpolynomially \textup{(}resp. exponentially\textup{)}, then every $\mu$-stationary state on $\textup{C}^*_r(\G)$ vanishes at $g$. 
\end{coro}

We should remark that our requirement on the rate of return of the random walk in the above statements is not a particularly strong condition. Indeed, in the case of a normal subgroup $H\triangleleft \G$ (that is, the random walk on the Cayley graph of the quotient group $\G/H$), the probability of return to the origin decays polynomially if and only if the quotient group $\G/H$ has polynomial growth. (See \cite{Woess}*{Theorem 6.6} for details.)

\medskip

Finally, note that, in view of \cite{KK14}, the above results give new sufficient conditions for the $\textup{C}^*$-simplicity of discrete groups with (sub-)rapid decay.

\subsection*{Traces on \texorpdfstring{$\textup{C}^*_r(\G)$}{C*r(G)}}
The problem of determining the support of traces $\tau$ on $\textup{C}^*_r(\G)$ is completely settled in \cite{BKKO2017}, where it was proven that every such $\tau$ is supported on the amenable radical $R(\G)$ of $\G$.

Recall that the \textit{amenable radical} of $\Gamma$ is the unique maximal normal amenable subgroup of $\Gamma$, which exists for every group $\G$ \cite{Day}. In \cite{Furm}*{Corollary 8}, Furman showed that the kernel of $\G\act \partial_F\G$ is indeed the amenable radical $R(\G)$.

Thus, $g\in \G$ acts non-trivially on $\partial_F\G$ $\iff$ $g$ does not belong to the amenable radical $\iff$ every trace $\tau$ on $\textup{C}^*_r(\G)$ vanishes on $g$.

In \cite{gong2015}*{Lemma 2.8}, Gong observed that \emph{if $\G$ has rapid decay, and $g\in \G$ is such that its conjugacy class grows superpolynomially, then every trace on $\textup{C}^*_r(\Gamma)$ vanishes on $g$}, 
and used that to give a complete description of traces on the reduced $\textup{C}^*$-algebras of hyperbolic groups. That, however, follows from results of \cite{BKKO2017} as we show below in Corollary~\ref{areqpc}.

In the following lemma, we record a more general version of \cite{gong2015}*{Lemma 2.8}.

\begin{lem}\label{lem:gong}
Let $\G$ be a finitely generated group. If $g\in\Gamma$ is such that 
$\rho^2=\underline{o}(\alpha_{\scriptscriptstyle \textup{Cl}(g)})$, then $\tau(\lambda_g)=0$ for every trace $\tau$ on $\textup{C}^*_r(\Gamma)$.
\end{lem}

\begin{proof}
Set $A_n = B_n(\G)\cap \text{Cl}(g)$, and $z_n = \sum_{h\in A_n}\lambda(h)$.
Let $\tau$ be a trace on $\textup{C}^*_r(\Gamma)$. We have
\[\begin{split}
|\tau(g)||A_n| &= \big|\sum_{h\in A_n}\tau(g)\big| = |\tau(z_n)|\leq \|z_n\|\\&\leq \rho(n)\|z_n\|_2 = \rho(n)|A_n|^{1/2},
\end{split}\]
for all $n\in\bN$, which implies $|\tau(g)|\le \liminf_n \rho(n) |A_n|^{-1/2}= 0$.
\end{proof}

\begin{cor}
If $\G$ has \textup{RD} \textup{(}respectively, \textup{SRD}\textup{)}, then every trace $\tau$ on $\textup{C}^*_r(\Gamma)$ vanishes on elements $g$ whose conjugacy class has superpolynomial growth \textup{(}respectively, exponential growth\textup{)}.
\end{cor}

\section{A growth criterion for freeness of boundary actions: Proof of Theorem~\ref{thm:freeness}}

In this section we prove Theorem~\ref{thm:freeness}.  
For the proof, we require the following generalization of \cite{Ken15}*{Proposition 3.2}. Given a compact space $\mathcal{X}$, we denote by $B_b(\mathcal{X})$ the commutative C$^*$-algebra of bounded Borel functions on $\mathcal{X}$, equipped with the supremum norm. We denote by $\Sub(\Gamma)$ the compact space of subgroups of $\Gamma$, equipped with the product (Chabauty) topology, and $\Sub_\text{a}(\Gamma)\subset \Sub(\Gamma)$ the closed subspace of amenable subgroups. The state space of $\textup{C}^*_r(\Gamma)$ is denoted by $\mathcal{S}(\textup{C}^*_r(\Gamma))$.

\begin{prop}\label{scarparo}
Let $\mathcal{X}$ be a compact $\Gamma$-space with amenable neighborhood stabilizers.
The map \,
$
\Gamma\ni g\mapsto \textup{\textbf{1}}_{{\rm int}(\textup{Fix}(g))}\in B_b(\mathcal{X})
$
extends to a unital positive $\Gamma$-equivariant map $\psi: \textup{C}^*_r(\Gamma)\to B_b(\mathcal{X})$.
\end{prop}
\begin{proof}
We reproduce the argument of \cite{Ken15}*{Proposition 3.2}. 
First, it is not hard to see that the map $\text{Stab}^o: \mathcal{X}\to \Sub(\Gamma)$ taking a point $x\in\mathcal{X}$ to the neighborhood stabilizer subgroup $\Gamma_x^o\in \Sub(\Gamma)$ is Borel.  
By assumption, $\text{Stab}^o$ takes values in $\Sub_\text{a}(\Gamma)$. Now recall that for each subgroup $H\in \Sub_\text{a}(\Gamma)$, the characteristic function $\chi_H: \Gamma\to \C$ extends to a state $\om_H: \textup{C}^*_r(\Gamma)\to \C$, and the mapping $\Sub_\text{a}(\Gamma)\to \mathcal{S}(\textup{C}^*_r(\Gamma))$ taking an amenable subgroup $H$ to the associated state $\om_H$ is continuous (when the codomain is equipped with the weak$^*$-topology.) In particular, the map $\mathcal{X}\to \mathcal{S}(\textup{C}^*_r(\Gamma))$, taking $\mathcal{X}\ni x\mapsto \om_{\Gamma_x}\in \mathcal{S}(\textup{C}^*_r(\Gamma))$, is Borel. Then, for each $a\in\mathcal{A}$, the map ${\psi}(a):\mathcal{X}\ni x\mapsto \om_{\Gamma_x}(a)\in\mathbb{C}$ is bounded  and Borel. The map $\psi : \textup{C}^*_r(\Gamma)\to B_b(\mathcal{X})$ is easily verified to be unital, positive, and $\Gamma$-equivariant, with $\psi(g) =  \textup{\textbf{1}}_{{\rm int}(\textup{Fix}(g))}$.
\end{proof}

Also recall that, given  $\mu\in\textup{Prob}(\Gamma)$ and a compact $\Gamma$-space $\mathcal{X}$, a measure $\nu\in \text{Prob}(\mathcal{X})$ is said to be $\mu$-stationary if $\nu = \sum_{h\in\Gamma}\mu(h)\,h \nu$, where $h\nu\in \pr(\mathcal{X})$ is defined by $h\nu(E):=\nu(h^{-1}E)$ for every $h\in \G$ and Borel subset $E\subseteq \mathcal{X}$. If $\mu$ is generating (i.e. its support generates $\G$ is a semigroup), then $\nu$ is non-singular (i.e. $g\nu$ and $\nu$ are in the same measure-class for every $g\in\G$), and the support of $\nu$ is a closed $\G$-invariant subset of $\mathcal{X}$. In particular, if $\mu$ is generating and the action $\G\act \mathcal{X}$ is minimal, then every $\mu$-stationary probability measure $\nu\in \pr(\mathcal{X})$ is non-singular and has full support. 

The $\text{C}^*$-dynamical analogue of a stationary measure was investigated in \cite{hart-kalantar-1}. A state $\tau$ on $\textup{C}^*_r(\G)$ is said to be $\mu$-\textit{stationary} if $\tau(a)=\sum_{t\in\G}\mu(t)\,\tau(\lambda_{t}a\lambda_{t^{-1}})$ for all $a\in \textup{C}^*_r(\G)$. We write $\mu * a = \sum_{t\in\Gamma}\mu(t)\lambda_ta\lambda_{t^{-1}}$, so that $\tau(a) = \tau(\mu * a)$ whenever $\tau$ is $\mu$-stationary.

It is a basic fact that $\mu$-stationary measures and states always exist.

\begin{lem}\label{lem:station-topfree}
Assume $g\in\G$ is such that for some generating measure $\mu\in\pr(\G)$, every $\mu$-stationary state $\tau$ on $\textup{C}^*_r(\G)$ vanishes on $g$. Then $g$ acts topologically freely on every compact minimal $\Gamma$-space $\mathcal{X}$ with amenable neighborhood stabilizers.
\begin{proof}
For contrapositive, assume there is a non-empty open set $U\subset \mathcal{X}$ that is pointwise fixed by $g$. Let $\mu\in\pr(\G)$ be generating, and let $\nu\in\pr(\mathcal{X})$ be $\mu$-stationary. Since $\nu$ has full support, $\nu(U)\neq 0$. Let $\psi: \textup{C}^*_r(\Gamma)\to B_b(\mathcal{X})$ be the unital positive $\Gamma$-equivariant map from Proposition~\ref{scarparo}. Then $\tau:=\nu\circ \psi$ is a $\mu$-stationary state on $\textup{C}^*_r(\Gamma)$, and we have $\tau(\lambda_g)= \nu({\rm int}({\rm Fix}(g)))\ge \nu(U)>0$. Hence, the claim follows.   
\end{proof}
\end{lem}

\medskip

\noindent
{\it Proof of Theorem~\ref{thm:freeness}.}~
If $g$ does not act topologically freely on $\mathcal{X}$, then its fixed point space has nonempty interior $V = {\rm int}(\text{Fix}(g))$. By minimality and compactness, there are $h_1, ..., h_m\in\Gamma$ such that $\bigcup_{j=1}^m h_jV = \mathcal{X}$ (note that $h_jV = {\rm int}(\text{Fix}(h_jgh_j^{-1}))$). For each $j$, we set $g_j = h_jgh_j^{-1}$, and $C_j = C_\G(g_j)=h_jCh_j^{-1}$. 

Furthermore, let $h_0\in\G$ be such that $h_0C\neq h_iC$ for any $i=1, \dots, m$. (Note that the growth condition \eqref{eqn} forces $C$ to have infinite index in $\G$.)
Denote $C_0 = h_0Ch_0^{-1}$.

By \eqref{eqn}, we may choose $N\in \bN$ so that for all $1\le i, j \le m$, $|g_i| < N$, and $\rho(3N)^2\cdot \alpha_{\scriptscriptstyle C_0\cap h_jh_i^{-1}C_i}(2N)< \alpha_{\scriptscriptstyle C_0}(N)/m^4$.

We observe that if $a,b\in C_0$ are such that $ag_ia^{-1}=bg_jb^{-1}$ for some $1\le i, j\le m$, then $b^{-1}a\in C_0\cap h_jh_i^{-1}C_i$.

In particular, for a fixed $a\in C_0\cap B_N(\G)$, the cardinality of the set $\{b\in C_0 \cap B_N(\G) \,:\, ag_ia^{-1}=bg_jb^{-1} ~\text{ for some }~ 1\le i, j\le m\,\}$ is bounded by $\sum_{i,j=1}^m |B_{2N}(\G)\cap C_0\cap h_jh_i^{-1}C_i|\le m^2\cdot  \max_{1\le i,j\le m} \alpha_{\scriptscriptstyle C_0\cap h_jh_i^{-1}C_i}(2N)=:M$. 

Choose a maximal set $A\subset B_N(\G)\cap C_0$ such that the maps $A\ni s\mapsto sg_is^{-1}$ are injective for all $i=1, \dots, m$. It follows that $|A|\ge |B_{N}(\G)\cap C_0|/M$.

Define $\mu\in\text{Prob}(\Gamma)$ by $\mu = |A|^{-1}\sum_{k\in A}\delta_k$. Let $\nu\in \text{Prob}(\mathcal{X})$ be a $\mu$-stationary measure. Composing $\nu$ with the unital positive $\Gamma$-equivariant map $\psi: \textup{C}^*_r(\Gamma)\to B_b(\mathcal{X})$ from Proposition \ref{scarparo}, we obtain a $\mu$-stationary state $\tau = \nu\circ\psi$ on $\textup{C}^*_r(\Gamma)$ such that $\tau(g) = \nu({{\rm int}(\textup{Fix}(g))})$, for all $g\in\Gamma$. For each $j = 1, ..., m$, we have
\begin{align*}
\nu(h_jV) &= \nu({{\rm int}(\textup{Fix}(g_j))}) = \tau(\lambda_{g_j}) = \tau(\mu * \lambda_{g_j}) \\&\leq \|\mu * \lambda_{g_j}\|= |A|^{-1}\cdot \big\|\sum_{t\in A}\lambda_{tg_jt^{-1}}\big\|
\\&\leq 
\rho(3N)\cdot |A|^{-1}\cdot \big\|\sum_{t\in A}\lambda_{tg_jt^{-1}}\big\|_2 = \rho(3N)\cdot  |A|^{-1/2} 
\\&\le \rho(3N)\cdot M^{1/2} \cdot |B_{N}(\G)\cap C_0|^{-1/2}
< 1/m\,.
\end{align*}
But then,
\[
1 = \nu(\mathcal{X}) = \nu\left(\bigcup_{j=1}^mh_jV\right) \leq \sum_{j=1}^m \nu(h_jV) < \sum_{j=1}^m \frac{1}{m} = 1,
\]
which is a contradiction.
\qed

\medskip

\textit{Proof of Corollary \ref{urs}.} By  \cite{urs}*{Corollary 1.2}, an amenable URS $\mathcal{H}\subset \Sub(\Gamma)$ can be realized as the family $\mathcal{H} = \{\Gamma_x\}_{x\in\mathcal{X}}$ of (amenable) stabilizers for a compact minimal $\Gamma$-space $\mathcal{X}$, which is everywhere continuous under the stabilizer map $\text{Stab}: \mathcal{X}\to \Sub(\Gamma)$. On such a $\Gamma$-space, topological freeness is equivalent to freeness, so $g\not\in \Gamma_x$ for all $x\in\mathcal{X}$.
\qed

\medskip

We prove Corollary~\ref{cor:srd} in the next section.

\section{Support of stationary states: Proof of Theorem~\ref{thm:Gong}}

We prove Theorem~\ref{thm:Gong} and Corollaries \ref{cor:srd} and \ref{cor:stn} in this section, in particular giving a description of support of stationary states on the reduced $\textup{C}^*$-algebras of groups with \textup{RD} or \textup{SRD}.

\medskip

\noindent
{\it Proof of Theorem~\ref{thm:Gong}.}
Let $\mu\in{\rm Prob}(\G)$ be symmetric with ${\rm supp}(\mu)\subset B_N(\G)$. 
Denote $q: \G\to \G/C_\G(g)$ for the canonical quotient map. 
Let $\ep>0$ be given. By the assumptions, we can choose $n\in\bN$ such that $\rho\big(2nN+|g|\big)\,\kappa_{2n}^{1/4}
\,<\,\ep$.
Let $\tau$ be a $\mu$-stationary state on $\textup{C}^*_r(\G)$. 
Then, we have
\begin{align*}
|\tau(\lambda_g)| &= \big|\tau\big(\sum_{t\in\Gamma}\mu^{*n}(t)\lambda_{tgt^{-1}}\big)\big| 
\\&\le \big\|\sum_{t\in\Gamma}\mu^{*n}(t)\lambda_{tgt^{-1}}\big\|
\\&\leq \rho\big(2nN+|g|\big)\,\big\|\sum_{t\in\Gamma}\mu^{*n}(t)\delta_{tgt^{-1}}\big\|_2
\\&=
\rho\big(2nN+|g|\big)\,\big\|q_*(\mu^{*n})\big\|_2
\\&\le
\rho\big(2nN+|g|\big)\,\big\|q_*(\mu^{*n})\big\|_\infty^{1/2} .
\end{align*}
Since the $\mu$-random walk on the Schreier graph $\G/C_\G(h)$ is symmetric, the probability $q_*(\mu^{*n})(hC_\G(g))$ of stepping into the coset $hC_\G(g)$ after $n$-steps is bounded above by $\sqrt{\kappa_{2n}}$.
Hence, combined with the above we get the estimates
\[\begin{split}
|\tau(\lambda_g)| &\le 
\rho\big(2nN+|g|\big)\,\big\|q_*(\mu^{*n})\big\|_\infty^{1/2}
\\ &\le \rho\big(2nN+|g|\big)\,\kappa_{2n}^{1/4}
\,<\,\ep .
\end{split}\]
It follows that $\tau(\lambda_g)=0$, and this completes the proof. 
\qed

\medskip

\noindent
{\it Proof of Corollary~\ref{cor:stn}.}~
Both cases follow immediately from Theorem~\ref{thm:Gong} since in both cases we have $\rho^k={o}(\kappa^{-1})$ for every $k\in\bN$. \qed

\begin{cor}\label{cor:Oles}
Assume $\G$ has {\rm SRD}. If $C_\G(g)$ is not co-amenable in $\G$ for an element $g\in \G$, then $g$ acts topologically freely on every compact minimal $\Gamma$-space $\mathcal{X}$ with amenable neighborhood stabilizers.\\ If, for every non-trivial $g\in\G$, $C_\G(g)$ is not co-amenable, then $\G$ is $\textup{C}^*$-simple.

\begin{proof}
Since $C_\G(g)$ is not co-amenable in $\G$, the Schreier graph $\G/C_\G(g)$ is non-amenable, and in particular, by the spectral gap, the probability of the return to the origin in the $\mu$-random walk decays exponentially. Thus, by  Corollary~\ref{cor:stn}, every $\mu$-stationary state on $\textup{C}^*_r(\G)$ vanishes at $g$. Hence, Lemma~\ref{lem:station-topfree} implies that $g$ acts topologically freely on every compact minimal $\Gamma$-space $\mathcal{X}$ with amenable neighborhood stabilizers.
\\
If, for every non-trivial $g\in\G$, $C_\G(g)$ is not co-amenable, then applying the above to the case $\mathcal{X}=\partial_F\G$, we get that $\G\act \partial_F\G$ is free, hence $\G$ is $\textup{C}^*$-simple by \cite{KK14}*{Theorem 6.2}.
\end{proof}
\end{cor}

The last assertion in the above corollary strengthens \cite{Olesen2016}*{Theorem 4.4.5}, where it was shown that if $\G$ has \textup{RD} and is not inner amenable, then $\G$ is $\textup{C}^*$-simple. 
Indeed, if $\G$ is not inner amenable, then $C_\G(g)$ is not co-amenable in $\G$ for every non-trivial $g\in\G$. 

\medskip

\noindent
{\it Proof of Corollary~\ref{cor:srd}.}~
If $C$ has subexponential growth, then it is amenable, hence not co-amenable in $\G$, and it follows by Corollary~\ref{cor:Oles} that $g$ acts freely on $\partial_F\G$, and we are done. So, we assume that $\alpha_{\scriptscriptstyle C}$ grows exponentially. Since $\alpha_{\scriptscriptstyle C\,\cap\, \g C\g^{-1}}$ grows subexponentially for all $\g\notin C$, and $\rho$ also grows subexponentially,  \eqref{eqn} obviously holds, hence we may invoke Theorem~\ref{thm:freeness} to conclude the proof. \qed

\medskip

In view of all the above, the following question is natural.

\begin{question}\label{Q}
Are there groups $\G$ with (sub-)rapid decay and trivial amenable radical that are not $\textup{C}^*$-simple? 
\end{question}

We see from Corollary~\ref{cor:Oles} that, for a group $\Gamma$ with SRD, the only possible obstruction to freeness of the action of $g$ on $\partial_F\G$ is amenability of the Schreier graph $\G/C_\G(g)$. 
A related notion to amenability of graphs is the \emph{Liouville property}, which asserts that every bounded harmonic function on the graph is constant. More precisely, in our context, given a subgroup $H\le \G$ and a measure $\mu\in\pr(\G)$, a function $f\in \ell^\infty(\G/H)$ is called $\mu$-harmonic if $f(gH)= \sum_{t\in\G} \mu(t) f(tgH)$ for every $g\in \G$. The subgroup $H$ is said to be \emph{co-Liouville} in $\G$ if there is a symmetric generating measure $\mu\in\pr(\G)$ such that every bounded $\mu$-harmonic function on $\G/H$ is constant.
By results of \cites{Furstenberg73, KaiVer, Rosn}, if $H\trianglelefteq\G$ is normal, then $H$ is co-amenable in $\G$ if and only if it is co-Liouville in $\G$. However, for general subgroups this equivalence fails \cites{Benj, Jus23}.

\begin{prop}
Let $\G$ be a countable group. If $g\in \G$ is such that $C_\G(g)$ is co-Liouville in $\G$, then $g$ acts on $\partial_F\G$ either freely or trivially.
\end{prop}
\begin{proof}
Assume $g\act \partial_F\G$ is not free, and let $V=\{x\in \partial_F\G : gx=x\}$ be the fixed point set of $g$, which is non-empty and clopen. Let $\mu\in\pr(\G)$ be a symmetric measure such that every bounded $\mu$-harmonic function on $\G/C$ is constant. Let $\nu\in\pr(\partial_F\G)$ be $\mu$-stationary, and define $f\in\ell^\infty(\G)$ by $f(g):=\nu(gV)$.
Denote $C=C_\G(g)$. Since $hV=V$ for every $h\in C$, it follows that $f$ is invariant under right translation by $C$, i.e. $f\in \ell^\infty(\G/C)$. Moreover, we have
$\sum_{t\in\G} \mu(t) f(tgC) = \sum_{t\in\G} \mu(t) \nu(tgV)
\overset{\ast}{=} \nu(gV) = f(gC)$,
where $\overset{\ast}{=}$ follows by $\mu$-stationarity of $\nu$. This shows $f$ is a bounded $\mu$-harmonic function on $\G/C$, hence constant by the assumption. Since $V$ is non-empty and $\nu$ has full support, $f$ is non-zero. 
Now, let $x\in \partial_F\G$, and find a net $(g_i)$ of elements $g_i\in \G$ such that $g_i\nu\to \delta_{x}$ weak*. We have $0\neq \nu(V) = f(e)= f(g_i^{-1}) = g_i\nu(V) \to \one_V(x)$, which implies $x\in V$. Hence, $V=\partial_F\G$, that is, $g$ acts trivially on $\partial_F\G$.
\end{proof}

As noted above, for a general subgroup $H\le \G$, being co-amenable and co-Liouville in $\G$ are not equivalent properties. However, it was proven in \cite{Hart-Tamuz} that any co-Liouville IRS is co-amenable. A converse result in our setting could offer a plausible path to a negative answer to Question~\ref{Q}. A relevant observation here is that for any subgroup $H\le\G$, the map $gH\to gHg^{-1}$ extends to a continuous $\G$-equivariant map $\beta( \G/H)\to {\rm Sub}(\G)$, and if $H$ is co-amenable in $\G$, the pushforward of the invariant mean defines an IRS on $\G$ supported on the closure of conjugates of $H$ in ${\rm Sub}(\G)$.

\section{Amenable radical and growth of conjugacy classes}

As mentioned in previous sections, \cite{gong2015}*{Lemma 2.8} (and more generally, Lemma~\ref{lem:gong}),
describes the support of traces on $\textup{C}^*_r(\G)$ in terms of the growth of conjugacy classes,
for $\G$ with (sub-)rapid decay. On the other hand, \cite{BKKO2017}*{Theorem 4.1} describes the support of traces on $\textup{C}^*_r(\G)$, for a general discrete group $\G$, in terms of the amenable radical. 
These facts raise the question of the relation between elements with ``slow growing'' conjugacy classes, and the amenable radical.

In this section, we investigate this, and in particular, we show that the above approaches are equivalent in the case of groups with (sub-)rapid decay. 

\begin{defn}
Let $\Gamma$ be a finitely generated group endowed with the word length. Define
\[
\textup{PC}(\Gamma) = \{g\in\Gamma: \alpha_{\scriptscriptstyle {\rm Cl}(g)} ~\text{ grows polynomially}\},
\]
and
\[
\textup{SEC}(\Gamma) = \{g\in\Gamma: \alpha_{\scriptscriptstyle {\rm Cl}(g)} ~\text{ grows subexponentially}\}.
\]
We say $\G$ is a {PC-group} (respectively, an {SEC-group}),
if $\Gamma = \textup{PC}(\Gamma)$ (respectively, $\Gamma = \textup{SEC}(\Gamma)$). 
\end{defn}

\begin{prop} 
The subgroups $\textup{PC}(\Gamma)\subseteq\textup{SEC}(\Gamma)$ are independent of the choice of the generating set defining the word length. Both subgroups are characteristic (hence, normal) in $\G$.
\end{prop}
\begin{proof}
Recall that if $S$ and $T$ are finite generating sets for $\Gamma$, then there is a constant $C > 0$ such that $C^{-1}|g|_S\leq |g|_T\leq C|g|_S$, for all $g\in\Gamma$. Now the first claim amounts to the observation that a function $n\mapsto f(n)$ has polynomial (resp. subexponential) growth if and only if $n\mapsto f(Cn)$ has polynomial (resp. subexponential) growth. The second claim follows from the observation that if $\varphi\in \text{Aut}(\G)$, then $S\subset\G$ generates $\G$ if and only if $\varphi(S)$ does.
\end{proof}

\begin{prop}\label{pcamenable}
Let $\G$ be a finitely generated group. 
\begin{itemize}
\item[(i)]
Let $h\in\G$. If 
$\alpha_{\scriptscriptstyle{\rm Cl}(h)}$ grows subexponentially \textup{(}respectively, polynomially\textup{)}, then so does $\beta_{\scriptscriptstyle C_\G(h)}$.
\item[(ii)]
Denote by $Z(\Gamma)$ the center of $\G$.
If $\G$ is a SEC-group \textup{(}respectively, PC-group\textup{)}, then the quotient group $\G/Z(\G)$ has subexponential \textup{(}respectively, polynomial\textup{)} growth.
\end{itemize}
\end{prop}

\begin{proof}
(i)\ 
Denote $H=C_\G(h)$, and for each $n\in\mathbb{N}$, set $\bar{B}_n 
= \{gH\in\G/H: |gH| \le n\}$, and $B_n 
= \{g\in\Gamma: |g| \le n\}$. 

The map $\Phi_n: \bar{B}_n\to \text{Cl}(h)\cap B_{2n+|h|}$ given by
$
\Phi_n(gH) = ghg^{-1}
$
is a well-defined injection for each $n\in\mathbb{N}$, and hence 
\[
|\bar{B}_n|\leq |\text{Cl}(h)\cap B_{2n+|h|}| ,
\]
which implies the assertions in the first item. 
\\[1ex]
(ii)\
Let $S\subset \G$ be a finite generating set.
We use the same notation as in the above case for $H=Z(\G)$, and define 
the maps $\Psi_n: \bar{B}_n\to \prod_{s\in S}\text{Cl}(s)\cap B_{2n+1}$ by
\[
\Psi_n(gH) = (g^{-1}sg)_{s\in S} .
\]
Then, $\Psi_n$
is a well-defined injection for every $n\in\bN$, and hence 
\[
|\bar{B}_n|\leq \prod_{s\in S}|\text{Cl}(s)\cap B_{2n+1}| ,
\]
which implies that if $\G$ is a SEC-group (respectively, PC-group), then $\G/Z(\G)$ is of subexponential (respectively, polynomial) growth.
\end{proof}

Note that part (ii) of Proposition~\ref{pcamenable} implies that every SEC-group (and hence every PC-group) is amenable. 
More generally, we have the following.

\begin{prop}\label{pcar}
If $\G$ is a finitely generated group, then 
$\textup{SEC}(\G)$ is amenable. In particular, $\textup{SEC}(\Gamma)\leq \textup{Rad}(\Gamma).$
\end{prop}

\begin{proof}
We show that every finitely generated subgroup of $\textup{SEC}(\G)$ is an SEC-group, hence amenable by Proposition~\ref{pcamenable}. It then follows that $\textup{SEC}(\G)$ is amenable. 

So, let $T\subseteq \textup{SEC}(\Gamma)$ be finite, and set $H = \langle T\rangle$. 
For every $h\in H$ we have
\[
B_n^T\cap \text{Cl}_H(h)\subseteq B_{Mn}^S\cap \text{Cl}_{\Gamma}(h),
\]
where $\text{Cl}_H(h)$ and $\text{Cl}_H(h)$ are the conjugacy classes of $h$ in $H$ and $\G$, respectively, and $M = \max\limits_{t\in T}|t|_S$. Thus,  
$H$ is an SEC-group with respect to $T$, as desired.
\end{proof}

\begin{cor}\label{areqpc}
If $\Gamma$ has \textup{SRD}, then $\textup{SEC}(\Gamma) = \textup{Rad}(\Gamma)$.
If $\Gamma$ has \textup{RD}, then $\textup{PC}(\Gamma) = \textup{SEC}(\Gamma) = \textup{Rad}(\Gamma)$.
\end{cor}

\begin{proof}
By Proposition~\ref{pcar} we have $\textup{PC}(\G)\leq \textup{SEC}(\G)\leq \textup{Rad}(\G)$. 

Conversely, since $\textup{Rad}(\G)$ is amenable, its characteristic function extends to a continuous trace on $\textup{C}^*_r(\G)$. Thus, Lemma~\ref{lem:gong} implies that if $\Gamma$ has \textup{SRD}, then $\textup{SEC}(\Gamma) = \textup{Rad}(\Gamma)$, and if $\Gamma$ has \textup{RD}, then $\textup{PC}(\Gamma) = \textup{Rad}(\Gamma)$.
\end{proof}

\begin{ex}
(i)\ The group $\G={\rm SL}_3(\bZ)$ does not have \textup{SRD}, but since its amenable radical is trivial, we have $\textup{PC}(\G) = \textup{SEC}(\G) = \textup{Rad}(\G)$.
\\
(ii)\
If $\G$ is an amenable group of exponential growth with trivial center, then $\G$ is not an SEC-group, thus
$\textup{SEC}(\G) \neq \textup{Rad}(\G)=\G$.
\\
(iii)\
If $\G$ is of intermediate growth with trivial center (e.g. the Grigorchuk group), then $\G$ is an SEC-group but not PC-group, and therefore we have 
$\textup{PC}(\G) \neq \textup{SEC}(\G)=\G$.
\end{ex}

\section{Le Boudec's groups}

In \cite{boudec}, Le Boudec constructs a family of countable automorphism groups on regular trees which have no nontrivial normal amenable subgroups, but are not C$^*$-simple. We briefly remind the reader of this construction. Fix a countable set $\Omega$ with at least three elements. Let $T_{\Omega}$ denote a regular tree with degree equal to the size of $\Omega$, and fix a coloring $c: \text{edge}(T_{\Omega})\to \Omega$ which is a bijection when restricted to the edges adjacent to any vertex $v\in \text{vert}(T_{\Omega}).$ Given $g\in \text{Aut}(T_{\Omega})$ and $v\in \text{vert}(T_{\Omega})$, let $\sigma_c(g, v)\in \text{Bij}(\Omega)$ denote the underlying permutation on $\Omega$ induced (by $c$) from the bijection between the adjacent edges of $v$ and the adjacent edges of $g(v)$.

Now fix subgroups $F< F'\leq \text{Bij}(\Omega)$ such that (1) $F$ acts freely on $\Omega$, (2) $F'\cdot \omega = F\cdot \omega$ for all $\omega\in\Omega$, and (3) the stabilizer subgroups $(F')_\omega\leq F'$ are amenable for all $\omega\in\Omega$. Le Boudec's group $G(F, F')$ is the set of automorphisms $\alpha\in \text{Aut}(T_{\Omega})$ such that (1) $\sigma_c(\alpha, v)\in F'$ for all $v\in \text{vert}(T)$ and (2) $\sigma_c(\alpha, v)\in F$ for all but finitely many $v\in \text{vert}(T)$. The group $G(F, F')$ admits a canonical boundary action on a certain compact space $\mathcal{X}$ of ends in $T_{\Omega}$; see \cite{boudec}*{Section 3.1} for details.

In this section, we apply our results from previous sections to conclude that certain subgroups of Le Boudec's groups $G(F, F')$ (including the groups $G(F, F')$ themselves) do not have sub-rapid decay. Given a vertex $v\in \textup{vert}(T_{\Omega})$, we will denote by $E(v)$ the set of edges adjacent to $v$. An edge $e\in \textup{edge}(T_{\Omega})$ disconnects $T_{\Omega}$ into two subtrees, called the $e$-\textit{half-trees} of $T_{\Omega}$.

\begin{lem}\label{amenable trees moves}
Let $\Omega, F, F',$ and $G = G(F, F')$ be as above. Suppose an automorphism $\alpha\in G$ has the following properties:
\begin{enumerate}
    \item $\alpha$ fixes an edge $e_{\alpha}\in \textup{edge}(T_{\Omega})$;
    \item If $\alpha$ fixes an edge $e'_{\alpha}$ adjacent to $v_{\alpha} = o(e_{\alpha})$, then $\alpha$ acts trivially on the $e'_{\alpha}$-half-tree of $T_{\Omega}$ which does not contain $v_{\alpha}$;
    \item $\sigma_c(\alpha, v_{\alpha})\neq \id$.
\end{enumerate}
Let $E\subsetneq E(v_{\alpha})$ denote the set of all edges adjacent to $v_{\alpha}$ which are fixed by $\alpha$. Let $T_\alpha\subset T_{\Omega}$ \textup{(}resp. $T_{\alpha}^c\subset T_{\Omega}$\textup{)} denote the maximal sub-tree of $E_{\Omega}$ with $v_{\alpha}\in \textup{vert}(T_{\alpha})$ and $E(v_{\alpha})\backslash E$ disjoint to $\textup{edge}(T_{\alpha})$ \textup{(}resp. $E$ disjoint to $\textup{edge}(T_{\alpha})$\textup{)}.
If $\beta \in C_G(\alpha)$, then there are $\varphi,\psi \in C_G(\alpha)$ such that $\varphi|_{T_\alpha} = \id|_{T_\alpha}$, $\psi|_{T_\alpha^c} = \id|_{T_\alpha^c}$, and $\beta = \varphi \circ \psi = \psi \circ \varphi$.
\end{lem}
\begin{proof}
The conditions on $\alpha$ imply that $T_{\alpha}$ and $T_{\alpha}^c$ are nonempty $\alpha$-invariant sub-trees of $T_{\Omega}$ with $T_{\alpha}\cap T_{\alpha}^c = \{v_{\alpha}\}$ and $T_{\alpha}\cup T_{\alpha}^c = T_{\Omega}$. The automorphism $\alpha$ restricts to a trivial action on $T_{\alpha}$ and a free action on $T_{\alpha}^c$.

Let $\beta \in C_G(\alpha)$. We claim that  $\beta(v_{\alpha}) = v_{\alpha}$. If not, then (since $\alpha$ acts freely on $T_{\alpha}^c$) there would exist a vertex $w \in \text{vert}(T_\alpha^c)$ such that $\beta(w) \in \text{vert}(T_\alpha)$. Again, since $\alpha$ acts freely on $T_\alpha^c$, we have that $\alpha(w) \neq w$. However,
\begin{equation*}
  \beta(\alpha(w)) = \alpha(\beta(w)) = \beta(w),
\end{equation*}
so $\alpha(w) = w$, a contradiction.

Consider the endomorphisms $\varphi$ and $\psi$ of $T_{\Omega}$ determined on $\text{vert}(T_{\Omega})$ by the formulas
\begin{equation*}
    \varphi(w) = \begin{cases}
        w & w \in \text{vert}(T_\alpha) \\
        \beta(w) & w \in \text{vert}(T_\alpha^c)
    \end{cases}\qquad\text{and}\qquad\psi(w) = \begin{cases}
        \beta(w) & w \in \text{vert}(T_\alpha) \\
        w & w \in \text{vert}(T_\alpha^c)
    \end{cases}.
\end{equation*}
Because $\alpha\beta = \beta\alpha$, $\beta(T_{\alpha}) = T_{\alpha}$, $\beta(T_{\alpha}^c) = T_{\alpha}^c$, and $\beta(v_{\alpha}) = v_{\alpha}$, it is immediate that the maps $\varphi$ and $\psi$ are well-defined automorphisms of $T_{\Omega}$ which commute with $\alpha$ and satisfy the relations $\beta = \varphi \circ \psi = \psi \circ \varphi$. 

It remains to see why $\varphi, \psi\in G$, but this is quickly observed: if $w \in \text{vert}(T_\alpha)$, then $\sigma_c(\varphi,w) = \sigma_c(\id,w)$. Thus, for such $w$, $\sigma(\varphi,w) \in F'$ and $\sigma(\varphi,w) \in F$ for all but finitely many $w\in \text{vert}(T_{\alpha})$. If instead $w \in \text{vert}(T_\alpha^c)$, then we have $\sigma(\varphi,w) = \sigma(\beta,w)$, so $\sigma(\varphi,w)\in F'$ and $\sigma(\varphi,w) \in F$ for all but finitely many $w\in \text{vert}(T_{\alpha})$. Since $\text{vert}(T_{\alpha})\cup \text{vert}(T_{\alpha^c}) = \text{vert}(T_{\Omega}),$ this shows that $\varphi \in G$. A similar argument shows that $\psi \in G$. Thus, $\varphi, \psi\in C_G(\alpha)$, as desired.
\end{proof}

\begin{lem}\label{amenable centralizer}
If $\alpha\in G$ satisfies the conditions in Lemma~\ref{amenable trees moves}, then $\textup{Fix}(\alpha)$ has non-empty interior and $C_G(\alpha)$ is amenable.

\begin{proof}
The first assertion is immediate from conditions (1) and (2) of Lemma~\ref{amenable trees moves}.

It follows from the last part of Lemma~\ref{amenable trees moves} that $C_G(\alpha)\leq \textup{Fix}(T_\alpha)\times \textup{Fix}(T_\alpha^c)$. By \cite{boudec}*{Proposition 5.4}, both groups $\textup{Fix}(T_\alpha)$ and $\textup{Fix}(T_\alpha^c)$ are amenable, hence $C_G(\alpha)$ is amenable.
\end{proof}
\end{lem}

\begin{thm}\label{lbgnotsrd}
Every Le Boudec's group $G = G(F, F')$ contains an element $\alpha$ satisfying the hypotheses of Lemma~\ref{amenable trees moves}.
If $H\leq G$ is a finitely generated non-amenable subgroup that contains such an element $\alpha$, and if there is a nonempty open set $U\subset \textup{Fix}(\alpha)\subset \partial T_\Omega$ such that $H$ acts minimally on $\overline{H\cdot U}\subset \partial T_\Omega$, then $H$ does not have \textup{(SRD)}.
\end{thm}
\begin{proof}

Consider an element $\alpha_0 \in G(F,F')$ which fixes an edge $e_{0}\in \text{edge}(T_{\Omega})$ and is such that $\sigma_c(\alpha, o(e_{0})) \neq \id$. Set $v_{0} = o(e_{0})$, and $E_0 = \{e \in E(v): \alpha_0(e) = e\}$. Because $\sigma_c(\alpha_0, v_0) \neq \id$, we have $E_\alpha \neq E(v_0)$. For every $e \in E(v_0)$, let $T_e$ be the $e$-half-tree of $T_{\Omega}$ which does not contain $v_{0}$. Consider the automorphism $\alpha$ of $T_{\Omega}$ determined on $\text{vert}(T_{\Omega})$ by the formula
\begin{equation*}
    \alpha(w) := \begin{cases}
        w & w \in \bigcup_{e \in E_0} \text{vert}(T_e) \\
        \alpha_0(w) & \text{otherwise}.
    \end{cases}
\end{equation*}
For every $w \in \bigcup_{e \in E_0} \textup{vert}(T_e)$, we have $\sigma_c(\alpha,w) = \sigma_c(\id,w) = \id \in F'$, and for all $w \not\in \bigcup_{e \in E_0} \textup{vert}(T_e)$, we have $\sigma_c(\alpha,w) = \sigma_c(\alpha_0,w) \in F'$. Similarly, $\sigma_c(\alpha,w)\in F'$ for all but finitely many $w\in \text{vert}(T_{\Omega})$, so that $\alpha\in G$. 

Now assume $H\le G$ is non-amenable with $\alpha\in H$, and assume that there is a nonempty open subset $U\subset \text{Fix}(\alpha)\subset \partial T_\Omega$ such that $H$ acts minimally on $\mathcal{X}:= \overline{H\cdot U}\subset \partial T_\Omega$. Since the stabilizer subgroups for the action $G\act \partial T_\Omega$ are all amenable, the same holds for the action $H\act \mathcal{X}$. By Lemma~\ref{amenable centralizer}, $C_H(\alpha)\le C_G(\alpha)$ is amenable. In particular, since $H$ is non-amenable, $C_H(\alpha)$ is not co-amenable in $H$. Hence, it follows from Corollary~\ref{cor:Oles} that $H$ does not have (SRD).
\end{proof}

\nocite{*}
\bibliographystyle{plain}
\bibliography{agkpreferences}

@article {BKKO2017,
    AUTHOR = {Breuillard, Emmanuel and Kalantar, Mehrdad and Kennedy,
              Matthew and Ozawa, Narutaka},
     TITLE = {\textup{C}$^*$-simplicity and the unique trace property for discrete groups},
   JOURNAL = {Publ. Math. Inst. Hautes \'Etudes Sci.},
  FJOURNAL = {Publications Math\'ematiques. Institut de Hautes \'Etudes
              Scientifiques},
    VOLUME = {126},
      YEAR = {2017},
     PAGES = {35--71},
      ISSN = {0073-8301,1618-1913},
   MRCLASS = {46L10 (20C07 37A55 46L89)},
  MRNUMBER = {3735864},
MRREVIEWER = {Anton\ Deitmar},
       DOI = {10.1007/s10240-017-0091-2},
       URL = {https://doi.org/10.1007/s10240-017-0091-2},
}

@article{gw_urs,
author = {Glasner, Eli and Weiss, Benjamin},
year = {2014},
month = {02},
pages = {},
title = {Uniformly recurrent subgroups},
isbn = {9781470409319},
doi = {10.1090/conm/631/12596}
}

@article{boudec,
author = {Le Boudec, Adrien},
year = {2016},
month = {12},
pages = {},
title = {\textup{C}$^*$-simplicity and the amenable radical},
volume = {209},
journal = {Inventiones mathematicae},
doi = {10.1007/s00222-016-0706-0}
}

@article{LeB-MBon18,
     author = {Le Boudec, Adrien and Matte Bon, Nicol\'as},
     title = {Subgroup dynamics and \textup{C}$^*$-simplicity  of groups of homeomorphisms},
     journal = {Annales scientifiques de l'\'Ecole Normale Sup\'erieure},
     pages = {557--602},
     year = {2018},
     publisher = {Soci\'et\'e Math\'ematique de France. Tous droits r\'eserv\'es},
     volume = {Ser. 4, 51},
     number = {3},
     doi = {10.24033/asens.2361},
     mrnumber = {3831032},
     language = {en},
     url = {https://www.numdam.org/articles/10.24033/asens.2361/}
}

@misc{Oz25,
      title={Proximality and selflessness for group \textup{C}$^*$-algebras}, 
      author={Narutaka Ozawa},
      year={2025},
      eprint={2508.07938},
      archivePrefix={arXiv},
      primaryClass={math.OA},
      url={https://arxiv.org/abs/2508.07938}, 
}

@article{hart-kalantar-1,
    AUTHOR = {Hartman, Yair and Kalantar, Mehrdad},
     TITLE = {Stationary $\textup{C}^*$-dynamical systems},
      NOTE = {With an appendix by Uri Bader, Hartman and Kalantar},
   JOURNAL = {J. Eur. Math. Soc. (JEMS)},
  FJOURNAL = {Journal of the European Mathematical Society (JEMS)},
    VOLUME = {25},
      YEAR = {2023},
    NUMBER = {5},
     PAGES = {1783--1821},
      ISSN = {1435-9855,1435-9863},
   MRCLASS = {46L05 (20C07 37A55 46L35 46L55)},
  MRNUMBER = {4592860},
MRREVIEWER = {Catalin\ Badea},
       DOI = {10.4171/jems/1225},
       URL = {https://doi.org/10.4171/jems/1225},
}

@article{KK14,
url = {https://doi.org/10.1515/crelle-2014-0111},
title = {Boundaries of reduced $\textup{C}^*$-algebras of discrete groups},
author = {Mehrdad Kalantar and Matthew Kennedy},
pages = {247--267},
volume = {2017},
number = {727},
journal = {Journal für die reine und angewandte Mathematik (Crelles Journal)},
doi = {doi:10.1515/crelle-2014-0111},
year = {2017},
lastchecked = {2025-06-23}
}

@article{Ken15,
  title={An intrinsic characterization of $\textup{C}^*$-simplicity},
  author={Matthew Kennedy},
  journal={Annales scientifiques de l'{\'E}cole normale sup{\'e}rieure},
  year={2015},
  url={https://api.semanticscholar.org/CorpusID:119613577}
}

@article{gong2015,
    author = {Gong, Sherry},
    title = {Finite part of operator {K}-theory for groups with rapid decay},
    journal = {Journal of Noncommutative Geometry},
    year = {2015},
    volume  = {9},
    number  = {3},
    pages   = {697--706},
    publisher = {European Mathematical Society}
}

@phdthesis{Olesen2016,
    author ={Olesen, Kristian K.} ,
    title = {Analytic Aspect of The {T}hompson Groups},
    school = {Department of Mathematical Sciences, University of Copenhagen},
    year = {2017}
}

@book{Woess, 
    place={Cambridge}, 
    series={Cambridge Tracts in Mathematics}, 
    title={Random Walks on Infinite Graphs and Groups}, 
    publisher={Cambridge University Press}, 
    author={Woess, Wolfgang}, 
    year={2000}, 
    collection={Cambridge Tracts in Mathematics}
}

@incollection {Furstenberg73,
    AUTHOR = {Furstenberg, Harry},
     TITLE = {Boundary theory and stochastic processes on homogeneous
              spaces},
 BOOKTITLE = {Harmonic analysis on homogeneous spaces ({P}roc. {S}ympos.
              {P}ure {M}ath., {V}ol. {XXVI}, {W}illiams {C}oll.,
              {W}illiamstown, {M}ass., 1972)},
    SERIES = {Proc. Sympos. Pure Math.},
    VOLUME = {Vol. XXVI},
     PAGES = {193--229},
 PUBLISHER = {Amer. Math. Soc., Providence, RI},
      YEAR = {1973},
   MRCLASS = {22E40 (60J50)},
  MRNUMBER = {352328},
MRREVIEWER = {R.\ G.\ Laha},
}

@article{Day,
author = {Day, Mahlon M.},
year = {1957},
month = {},
pages = {509-544},
title = {Amenable Semigroups},
volume = {1},
journal = {Illinois Journal of Mathematics},
doi = {https://doi.org/10.1215/ijm/1255380675}
}

@article{Furm,
author = {Furman, Alex},
year = {2003},
month = {},
pages = {173-187},
title = {On minimal, strongly proximal actions of locally compact groups},
volume = {136},
journal = {Israel Journal of Mathematics},
doi = {10.1007/BF02807197}
}

@article{joli,
 ISSN = {00029947, 10886850},
 URL = {http://www.jstor.org/stable/2001458},
 author = {Paul Jolissaint},
 journal = {Transactions of the American Mathematical Society},
 number = {1},
 pages = {167--196},
 publisher = {American Mathematical Society},
 title = {Rapidly Decreasing Functions in Reduced $\textup{C}^*$-Algebras of Groups},
 urldate = {2025-06-11},
 volume = {317},
 year = {1990}
}

@misc{KS2025,
      title={More examples of additivity violation of the regularized minimum output entropy in the commuting-operator setup}, 
      author={Mehrdad Kalantar and Homayoon Shobeiri},
      year={2025},
      eprint={2501.15462},
      archivePrefix={arXiv},
      primaryClass={math.OA},
      url={https://arxiv.org/abs/2501.15462}, 
}

@article {Jus23,
    AUTHOR = {Juschenko, Kate},
     TITLE = {Liouville property of strongly transitive actions},
   JOURNAL = {Proc. Amer. Math. Soc.},
  FJOURNAL = {Proceedings of the American Mathematical Society},
    VOLUME = {151},
      YEAR = {2023},
    NUMBER = {9},
     PAGES = {4047--4053},
      ISSN = {0002-9939,1088-6826},
   MRCLASS = {60G50 (20P05 60B15 82B41)},
  MRNUMBER = {4607647},
MRREVIEWER = {J.\ A.\ van Casteren},
}

@article {Hart-Tamuz,
    AUTHOR = {Hartman, Yair and Tamuz, Omer},
     TITLE = {Stabilizer rigidity in irreducible group actions},
   JOURNAL = {Israel J. Math.},
  FJOURNAL = {Israel Journal of Mathematics},
    VOLUME = {216},
      YEAR = {2016},
    NUMBER = {2},
     PAGES = {679--705},
      ISSN = {0021-2172,1565-8511},
   MRCLASS = {22D05 (22E46 22F10 37A15)},
  MRNUMBER = {3557462},
MRREVIEWER = {George\ A.\ Willis},
       DOI = {10.1007/s11856-016-1424-4},
       URL = {https://doi.org/10.1007/s11856-016-1424-4},
}

@book {Benj,
    AUTHOR = {Benjamini, Itai},
     TITLE = {Coarse geometry and randomness},
    SERIES = {Lecture Notes in Mathematics},
    VOLUME = {2100},
      NOTE = {Lecture notes from the 41st Probability Summer School held in
              Saint-Flour, 2011,
              Chapter 5 is due to Nicolas Curien, Chapter 12 was written by
              Ariel Yadin, and Chapter 13 is joint work with Gady Kozma,
              \'Ecole d'\'Et\'e{} de Probabilit\'es de Saint-Flour.
              [Saint-Flour Probability Summer School]},
 PUBLISHER = {Springer, Cham},
      YEAR = {2013},
     PAGES = {viii+129},
      ISBN = {978-3-319-02575-9; 978-3-319-02576-6},
   MRCLASS = {05-06 (05C10 05C25 05C50 05C62 05C80)},
  MRNUMBER = {3156647},
       DOI = {10.1007/978-3-319-02576-6},
       URL = {https://doi.org/10.1007/978-3-319-02576-6},
}

@article{urs,
    title={Realizing uniformly recurrent subgroups}, 
    volume={40}, 
    DOI={10.1017/etds.2018.47}, 
    number={2}, 
    journal={Ergodic Theory and Dynamical Systems}, 
    author={MATTE BON, NICOLÁS and TSANKOV, TODOR}, 
    year={2020}, 
    pages={478–489}
}

@article {KaiVer,
    AUTHOR = {Kaimanovich, V. A. and Vershik, A. M.},
     TITLE = {Random walks on discrete groups: boundary and entropy},
   JOURNAL = {Ann. Probab.},
  FJOURNAL = {The Annals of Probability},
    VOLUME = {11},
      YEAR = {1983},
    NUMBER = {3},
     PAGES = {457--490},
      ISSN = {0091-1798,2168-894X},
   MRCLASS = {60B15 (43A07 60J15)},
  MRNUMBER = {704539},
MRREVIEWER = {Yves\ Derriennic},
       URL =
              {http://links.jstor.org/sici?sici=0091-1798(198308)11:3<457:RWODGB>2.0.CO;2-5&origin=MSN},
}

@article {Rosn,
    AUTHOR = {Rosenblatt, Joseph},
     TITLE = {Ergodic and mixing random walks on locally compact groups},
   JOURNAL = {Math. Ann.},
  FJOURNAL = {Mathematische Annalen},
    VOLUME = {257},
      YEAR = {1981},
    NUMBER = {1},
     PAGES = {31--42},
      ISSN = {0025-5831,1432-1807},
   MRCLASS = {43A05 (28C10 60B15 60J15)},
  MRNUMBER = {630645},
MRREVIEWER = {Yves\ Guivarc'h},
       DOI = {10.1007/BF01450653},
       URL = {https://doi.org/10.1007/BF01450653},
}

@misc{SriGreg-SRD,
      title={Some remarks on decay in countable groups and amalgamated free products}, 
      author={Srivatsav Kunnawalkam Elayavalli and Gregory Patchell and Lizzy Teryoshin},
      year={2025},
      eprint={2509.08754},
      archivePrefix={arXiv},
      primaryClass={math.GR},
      url={https://arxiv.org/abs/2509.08754}, 
}

\end{document}